\documentclass{article}

\usepackage{caption}
\usepackage{longtable}
\usepackage{cellspace}
\usepackage{booktabs}
\usepackage{arydshln}
\usepackage{amsfonts}
\usepackage{amssymb}
\usepackage{graphicx}
\usepackage{amsmath}
\usepackage{amsthm}
\usepackage{amsmath}
\usepackage{amsfonts}
\usepackage{amsthm}
\usepackage{mathabx}
\usepackage[colorlinks=false,linkcolor=blue,citecolor= red,bookmarksopen=true]{hyperref}
\usepackage{comment}
\usepackage{enumerate}
\usepackage{dsfont}

\newcommand{\R}{\mathbb{R}} 

\newcommand{\ind}{\mathbf{1}}

\newcommand{\PW}{\mathfrak{p}}

\newcommand{\CF}{\mathcal F}

\newcommand{\Fcal}{{\mathcal F}}

\newcommand{\Hcal}{{\mathcal H}}

\newcommand{\Mcal}{{\mathcal M}}

\newtheorem{theorem}{Theorem}[section]
\newtheorem{proposition}[theorem]{Proposition}

\newtheorem{lemma}[theorem]{Lemma}

\newtheorem{remark}[theorem]{Remark}
\newtheorem{example}[theorem]{Example}
\newtheorem{definition}[theorem]{Definition}

\begin{document}

\title{The birth of (a robust) Arbitrage Theory in de Finetti's early contributions}
\date{\today}
\author{Marco Maggis\thanks{
Dipartimento di Matematica, Universit\`a degli Studi di Milano, Via Saldini
50, 20133 Milano, Italy, \emph{marco.maggis@unimi.it}. 
\\ The author would like to thank Teddy Seidenfeld for pointing their contribution  \cite{SSK08}, during a de Finetti Risk seminar held in Milan in April 2015. Moreover the author is grateful to Matteo Burzoni, Marco Frittelli and the students of the classes of Mathematical Economics for stimulating discussions on this topic.}}

\maketitle

\begin{abstract} \noindent \textit{Il significato soggettivo della probabilit\`a} (1931) by B. de Finetti \cite{deF} is unanimously considered the rise of `subjectivism', a notion which strongly influenced both Probability and Decision Theory. 
What is less acknowledge is that \cite{deF} posed the foundations of modern arbitrage theory. In this paper we aim at examining how de Finetti's contribution should be considered as the precursor of Asset Pricing Theory and we show how his findings relate to recent developments in Robust Finance.    
\end{abstract}

\noindent \textbf{Keywords}: subjective probability, coherent prevision, arbitrage opportunity, linear pricing rule 

\parindent=0em \noindent

\section{Introduction}
In the economical and financial literature the \emph{Asset Pricing Model} postulates that
prices of financial assets composing the market are known at a certain initial time, while prices at future times are modelled adopting random outcomes (usually stochastic processes). The key reason for producing such models is to assign
rational prices to contracts which are not liquid enough to have a
market--determined price. The economical rationale beyond this approach is the
principle of absence of arbitrage opportunities, stating that it should not be possible to make a riskless profit by exploiting price discrepancies between different markets or financial instruments. 
If an arbitrage opportunity were to arise, rational investors would quickly exploit it, bringing prices back to equilibrium. Starting from these premises, the theory of pricing
by no arbitrage has been successfully developed (both for discrete and continuous time models) leading to the formulation of the so-called Fundamental Theorems of Asset Pricing
(FTAP), which establish equivalence between absence of arbitrage and existence
of risk neutral pricing rules (e.g. \cite{BlackScholes:73, DMW90, DS:FTAP, HK79, Ross76}). 

\medskip

W. Schachermayer \cite{Sch2010} already noticed that such fundamental questions date back much further in the literature and in particular to the contributions of \cite{Ke55,Sh55}. Nevertheless both these papers only provide logical formalization of concepts that were previously introduced and developed by B. de Finetti in \cite{definetti}. The relation between de Finetti \emph{subjective probability} and the notion of \emph{fair evaluation} of bets is partially unknown to the wide audience whose research activity concerns mathematical modelling for  Finance. This is mostly due to the fact that \cite{definetti} is written in Italian and composed as a discursive argumentation rather than adopting a mathematical formalization.  
\\ In de Finetti's perspective the estimation of the probability of events which influence everyday
life and decisions is a very hard task, which is subject to the personal expertise and affected by mistakes. This motivates the famous assertion which can be found in the preface of the
book Theory of Probability \cite{deF}, which summarizes the
trajectory of the research agenda he followed since the very beginning of his career. 

\medskip

\textquoteleft \textquoteleft My thesis, paradoxically, and
a little provocatively, but nonetheless genuinely, is simply this: PROBABILITY DOES NOT EXISTS.
The abandonment of superstitious beliefs about the existence
of Phlogiston, the Cosmic Ether, Absolute Space and Time, ... , or
Fairies and Witches, was an essential step along the road to
scientific thinking. \\Probability, too, if regarded as something
endowed with some kind of objective existence, is no less a
misleading misconception, an illusory attempt to exteriorize or
materialize our true probabilistic beliefs.
\textquoteright\textquoteright

\medskip

Nevertheless there exists a common agreement that the axioms characterizing (finitely additive) probability measures are the minimal necessary requirement to construct a reasonable mathematical theory regarding decisions under uncertainty. This evidence can find an unexpected motivation if we regard axioms of probability in relation to fairness of bets between two agents. Indeed the starting point to formulate the principle of subjective probability in \cite{deF} can be depicted by the following quotation\footnote{This is a free translation of the original Italian excerpt \cite[Page 154]{definetti}: \emph{\textquoteleft\textquoteleft Pu\`o sembrare infatti che nell'atto di stabilire le condizioni
di una scommessa influiscano su di noi piuttosto l'amore e il
timore del rischio o simili circostanze del tutto estranee che non
quel grado di fiducia che corrisponde alla nozione pi\`u o meno
intuitiva di probabilit\`a, e che noi ci proponiamo di misurare.
\\ Ci\`o sarebbe evidentemente vero se si trattasse di fare una scommessa singola e ben determinata;
non lo \`e pi\`u invece se ci mettiamo nelle condizioni supposte:
di un individuo che debba tenere un banco di scommesse su dati
eventi, accettando alle stesse condizioni qualunque scommessa
nell'uno o nell'altro senso. Vedremo che egli \`e costretto allora
a rispettare certe restrizioni, che sono i teoremi del calcolo
delle probabilit\`a. Altrimenti egli pecca di coerenza, e perde
sicuramente, purch\`e l'avversario sappia sfruttare il suo errore.
\\ Un individuo che non commette un tale errore, che valuta cio\`e delle probabilit\`a
in modo da non mettere in grado i competitori di vincere a colpo
sicuro, lo diremo coerente. E il calcolo delle probabilit\`a non
\`e allora se non la teoria matematica che insegna ad essere
coerenti.\textquoteright \textquoteright}}:

\medskip

\textquoteleft \textquoteleft It may seem, in fact, that in the act of establishing the conditions of a bet, what influences us more are love and fear of risk or similar circumstances completely unrelated to that degree of trust corresponding to the more or less intuitive notion of probability that we intend to measure.
\\ This would be evidently true if we were playing a single and well determined
bet; it is no longer the case if we place ourselves in the supposed conditions:
an individual who must operate as a bookmaker for given events, accepting any bet under the same conditions, whether in one direction or the other. We
shall see that he is then forced to comply with certain restrictions, which
are the theorems of Probability. Otherwise he will fail to be
coherent, and will surely face a loss, as far as his opponent will
be able to exploit his mistake.
\\ An individual who does not make such a mistake, that is, who assesses probabilities in a way that does not enable competitors to win with certainty, will be called coherent.
And probability calculus is nothing but the mathematical theory that teaches us to be coherent.
\textquoteright \textquoteright  

\medskip

We can therefore consider de Finetti as the precursor of the modern Arbitrage Theory, since he was the first to connect the possibility of a sure gain to a notion of fair pricing. The astounding capability of de Finetti in foreseeing such a central aspect of modern Finance was already pointed out in \cite{Nau01}: 

\medskip

\textquoteleft \textquoteleft In the 1970’s – the so-called ‘golden age’ of asset pricing theory
– there was an explosion of interest among finance theorists in models
of asset pricing by arbitrage. The key discovery of this period
was the fundamental theorem of asset pricing (Ross, 1976; see
also Dybvig and Ross, 1987; Duffie, 1996). The theorem states that
there are no arbitrage opportunities in a financial market if and only
if there exists a probability distribution with respect to which the
expected value of every asset’s future payoffs, discounted at the
risk free rate, lies between its current bid and ask prices; and if
the market is complete, the distribution is unique. This is just de
Finetti’s fundamental theorem of subjective probability, with discounting
thrown in, although de Finetti is not usually given credit in
the finance literature for having discovered the same result 40 years
earlier.  \textquoteright \textquoteright

\medskip

This note is structured as follows: in Section \ref{1931} we provide a review of the contents of \cite{definetti} adopting a modern mathematical formalism. This contributes in providing an elegant introduction to the spring of Arbitrage Theory and its relation to the fair evaluation of bets, which could be appreciated by the mathematical international community and can be also adopted as a teaching tool. In Section \ref{1980} 
we will connect such seminal contributions to important developments which are 
extensively treated in the book Theory of Probability \cite{deF} (whose first Italian edition dates back to 1970), where the study is moved far beyond simple bets. In Theorem \ref{pricing:measure} and Theorem \ref{fund:prev} we also elaborate how de Finetti theory was indeed few steps away from reaching the summit of a robust version of the first Fundamental Theorem of Asset pricing and the pricing-hedging duality in incomplete markets. In particular we show that coherency is sufficient for finding a finitely additive pricing measure, but does not exclude existence of strong arbitrage opportunities (i.e. a strategy which gives a positive value in every possible state of the world), case in which such a measure cannot be countably additive.

\section{De Finetti 1931 \cite{deF} and a primeval concept of arbitrage opportunity}\label{1931}

We start this section introducing the basic setup considered in \cite{deF}.  
\\Let $\CF$ be a sigma algebra \footnote{Most of the reasoning in this paper only necessitates $\CF$ to be an algebra. Nevertheless we prefer to work under this slightly stronger assumption in order to apply some duality results in the last part of the paper.} on a space of events $\Omega$ and $\PW:\CF\to \R$ a
set function. For any event $A$ we denote
by $\ind_A:\Omega\to \R$ the function valued $1$ if
$\omega$ belongs to $A$ and $0$ otherwise. 
\\The function $\PW$ needs to be regarded as the belief of a bookmaker
$\mathfrak{bm}$ who is providing quotations for bets on events $A\in \CF$
and is willing to accept any betting in both directions. More precisely the bookmaker is characterized by the following type of strategies: $\mathfrak{bm}$ can receive any monetary multiple of $\PW(A)$, say $\alpha \cdot\PW(A)$ dollars, for \textquoteleft selling\textquoteright (in which case $\alpha>0$) or \textquoteleft buying\textquoteright (in which case $\alpha<0$) a bet which pays
$\alpha$ dollars  if $A$ occurs and $0$ otherwise. In this way the bookmaker's final random payoff is given
by 
$$\Omega\ni\omega\mapsto \alpha\cdot(\PW(A)-\mathbf{1}_A(\omega)).$$ 
An event $A\in\CF$ is relevant (or non-negligible) for a bookmaker $\mathfrak{bm}$ if $\PW(A)>0$. On the contrary an event $A\in\CF$ is negligible if $\PW(A)=0$.  

\medskip

The following definition is the core of the characterization of Subjective Probability: it states that a bookmaker is assessing the value of a bet in a coherent way if it is not possible to combine bets in order to obtain a sure gain.
\begin{definition}\label{coherent:M} The bookmaker $\mathfrak{bm}$ adopting a set function $\PW:\CF\to\R$  is coherent if and only if  for any family of relevant events $\{A_i\}_{i=1,\ldots,n}$ and $(\alpha_1,\ldots,\alpha_n)\in \R^n$ there does not exist any $\varepsilon>0$ such that
\begin{equation}\label{point:arb} \sum_{i=1}^{n} \alpha_i\cdot(\PW(A_i)-\mathbf{1}_{A_i}(\omega))\geq \varepsilon > 0 \quad \forall\,\omega\in\Omega. 
\end{equation}
\end{definition}

\medskip

From now on we denote by $\mathcal{P}_f$ the convex set of finitely additive probability measures i.e. set functions $\PW:\CF\to \R$ such that the following conditions hold
\begin{itemize}
\item $\PW(A)\geq 0$ for every $A\in \CF$;
\item $\PW(\Omega)=1$;
\item $\PW(A\cup B)=\PW(A)+\PW(B)$ for any $A,B\in\CF$ such that $A\cap
B=\varnothing$.
\end{itemize}
\medskip

We are now ready to state and prove that the absence of opportunities like those in Eq. \eqref{point:arb} is equivalent to the axioms of probability for $\PW$. In the literature $\PW$ is referred as Subjective Probability since it reflects the personal expertise of the bookmaker in assessing the value of the bet $\ind_A$. 

\medskip

\begin{theorem}[de Finetti 1931]\label{DF31} Let $\PW:\CF\to \R$ be a set function on an algebra of events $\CF$ chosen by a bookmaker $\mathfrak{bm}$ to price bets of the form $\{\ind_A\mid A\in \CF \}$. The bookmaker is coherent if and only if $\PW\in \mathcal{P}_f$.
\end{theorem}

\medskip

To understand how Theorem \ref{DF31} can be seen as a prototype of the first FTAP we make a preliminary comparison with the statement in \cite{HK79}. In Section \ref{1980} we shall deepen the discussion and show de Finetti's results should be rather compared with recent developments in Robust Finance (e.g. \cite{Bu+19}).  In the case of a finite set of scenarios $\Omega=\{\omega _{1},\ldots ,\omega_{n}\}$, let $s=(s^{1},\ldots ,s^{d})$ be the initial prices of $d$
risky assets with random outcome $S(\omega )=(S^{1}(\omega ),\ldots ,S^{d}(\omega
))$ for any $\omega \in \Omega$ and for simplicity assume that the numeraire asset is $s^0=1$ and $S^0=1$. Then, we have the following equivalence
\begin{itemize}
    \item No-Arbitrage condition: $\nexists H\in \mathbb{R}^{d}$ such that $H\cdot s\leq 0$  and $H\cdot S(\omega )\geq 0$ for any $\omega \in \Omega$, with $>$ for at least one $\omega \in \Omega$;
    \item $\exists Q\in \mathcal{P}_f$ such that $Q(\omega _{j})>0$ and $E_{Q}[S^{i}]=s^{i}$ , for all $\ 1\leq j\leq n,1\leq i\leq d$.
\end{itemize}
We stress that in both situations of Theorem \ref{DF31} and the FTAP in \cite{HK79} no reference probability measure is needed \emph{a priori}, but the existence of a pricing measure is rather deduced from the economic principle of coherency (resp. no arbitrage).  
Following Theorem \ref{DF31} a coherent bookmaker is assessing the values of any combination of bets following a \textquoteleft martingale principle\textquoteright, namely for any simple function $f=\sum_{i=1}^{N} \alpha_{i}\ind_{A_i}$, where $\{A_i\}_{i=1}^{N}$ is a partition of $\Omega$, we obtain the \textquoteleft fair price\textquoteright of $f$ by
$$ \text{price}(f) = \sum_{i=1}^{N}\alpha_i \cdot \text{price}(\mathbf{1}_{A_i}) =  \sum_{i=1}^{N}\alpha_i \cdot \PW(A_i)= E_{\PW}\left[f\right].$$
Nevertheless the notion of coherency is significantly weaker than the No-Arbitrage condition and leads to weaker properties on the pricing functional, as we shall prove in the next section (see Theorem \ref{pricing:measure}).   

\begin{remark}
    Even though $E_{\PW}$ is not a standard Lebesgue integral, as $\PW$ is only a finitely additive measure, this is not a technical issue as far as we restrict $E_{\PW}$  to simple functions (see \cite[Section 11.1]{AlBor}). 
\end{remark}

\begin{remark}\label{mon} Notice that condition $\PW(A\cup B)=\PW(A)+\PW(B)$ for any $A,B\in\CF$ with $A\cup B=\varnothing$ implies   $\PW(A)\leq \PW(B)$ for any $A\subset B$.
\\ Such monotonicity principle reflects automatically on the monotonicity of the functional $E_{\PW}$ on simple functions. 
\end{remark}

\medskip

\subsection{Proof of Theorem \ref{DF31}} It is important to highlight that the arguments that we are going to present recollect and synthesize some reasoning which can be found in \cite{definetti}. Nevertheless we could not find in the literature a reader-friendly complete proof of such a result and for completeness and clarity we decided to illustrate it in this preliminary part of the paper.        

\medskip

$(\mathbf{\Rightarrow})$ Let $\mathfrak{bm}$ be coherent. We first show that for any $A\in\CF$ we have $\PW(A)\geq 0$\footnote{Observe that positivity is linked to the initial convention for which if $\alpha>0$ then $\alpha\cdot\PW(A)$ is the cash received from $\mathfrak{bm}$ for selling bet $A$.}. Indeed if by contradiction $\PW(A)<0$ for some $A\in\CF$, then a player could pay the amount $\PW(A)$ to acquire the bet $\ind_A$. The player's  payoff would be $\ind_A-\PW(A) \geq -\PW(A)>0$, which contradicts the coherence of $\mathfrak{bm}$.  Similarly if $\PW(\Omega)<1$ then $\mathfrak{bm}$ is forced to accept a bet for which he receives the amount $\PW(\Omega)<1$ and pays $1$ whatever event $\omega\in\Omega$ occurs. Therefore the strategy $\alpha=-1$ violates coherence as $-\PW(\Omega)+1=\varepsilon>0$.
\\ Assume now that there existed a disjoint couple $A,B\in \CF$ such that $\PW(A\cup B)> \PW(A)+\PW(B)$. Consider the following bets:
\begin{enumerate}
\item pay $\PW(A)$ and receive $1$ if $A$ occurs, i.e. the payoff is $\ind_A-\PW(A)$; 
\item pay $\PW(B)$ and receive $1$ if $B$ occurs, i.e. the payoff is $\ind_B-\PW(B)$; 
\item receive $\PW(A\cup B)$ and pay $1$ if $A\cup B$ occurs, i.e. the payoff is $\PW(A\cup B)-\ind_{A\cup B}$.
\end{enumerate}
The aggregated strategy turns out to be  
$$\PW(A\cup B)-\ind_{A\cup B}-\PW(A)+\ind_A-\PW(B)+\ind_B=\PW(A\cup B)-\PW(A)-\PW(B)>0,$$ 
and in such a case the bookmaker $\mathfrak{bm}$ would not be coherent.
\\ In a similar way we can consider the case for a disjoint couple $A,B\in \CF$, such that $\PW(A\cup B)< \PW(A)+\PW(B)$.

\medskip

$(\mathbf{\Leftarrow})$ For the reverse implication let $\PW$ be a finitely additive probability measure. Assume by contradiction there exist a family $\{A_i\}_{i=1}^n\subseteq \CF$ and real numbers $\{\alpha_i\}_{i=1}^{n}$ such that \[f(\omega)=\sum_{i=1}^{n} \alpha_i\cdot(\PW(A_i)-\mathbf{1}_{A_i}(\omega))\geq \varepsilon > 0 \quad \forall\,\omega\in\Omega.\] 
Then if we compute $E_{\PW}[f]$ we get $\varepsilon\leq E_{\PW}[f]$, which is clearly a contradiction as $E_{\PW}[f]=\sum_{i=1}^{n} \alpha_i\cdot(\PW(A_i)-E_{\PW}[\mathbf{1}_{A_i}])=0$.

\section{From simple bets to general gambles}\label{1980} Throughout his intensive research activity de Finetti further developed his original intuition, extending the analysis from simple bets to general bounded random variables. We start this section by pointing out some key aspects which can be found in \cite{deF} and are also commented in \cite{SSK08}. 
\\ Our primary object so far was a collection of bets of the form $\{\ind_A\mid A\in\CF\}$ and their linear combinations. We now move to a general set of gambles $\mathcal{H}\subset \mathcal{L}^{\infty}:=\{f:\Omega\to \R\mid \CF\text{-measurable and bounded}\}$. We can imagine $\mathcal{H}$ as a liquid financial market where it is possible to invest (buying or selling) without specific constraints.  We only assume that 
$\mathbf{1}_{\Omega}\in\mathcal{H}$ (which is to be interpreted as the zero coupon bond with $0$ interest rate). We shall denote by $\text{span}(\Hcal)$ the linear space generated by $\Hcal$.
\\ The topological dual of $\mathcal{L}^{\infty}$ is the space of charges with bounded variation (\cite[Theorem 14.4]{AlBor}) which shall be denoted by $\mathbf{ba}$. Positive charges are denoted by $\mathbf{ba}_+$ and  the convex set $\mathcal{P}_f$ is included in $\mathbf{ba}_+ $.  For any $\mu\in \mathbf{ba}_+$ the action on $f\in \mathcal{L}^{\infty}$ is described by $\int_{\Omega}fd\mu$ where the integral is constructed in \cite[Section 11.2]{AlBor} \footnote{ Here for simplicity we only recall that $$\int_{\Omega}fd\mu =\sup\left\{\int_{\Omega}g d\mu \mid g\text{ simple and } g\leq f\right\}$$ }.

\medskip

We assume the existence of a functional $\pi:\mathcal{H}\to \R$ which in de Finetti words \cite{deF} has the following meaning: 

\medskip 

 \textquoteleft\textquoteleft The function $\pi$ represents the opinion
 of an individual who is faced with a situation of uncertainty. To each random magnitude $f$, there corresponds the
 individual's evaluation $\pi(f)$, the prevision of $f$, whose meaning, operationally, reduces, in terms of gain, to that of the (fair) price
 $f$.  \textquoteright \textquoteright

\medskip

We can therefore easily adapt Definition \ref{coherent:M}  to this more general situation.

\medskip

\begin{definition} $\pi:\mathcal{H}\to \R$ is \textbf{coherent} if for every
$n$, $f_1,\dots, f_n \in \mathcal{H}$ and
$\beta_1,\dots,\beta_n\in\R$ we have
$$\inf_{\omega\in \Omega}\sum_{i=1}^n\beta_i[\pi(f_i)-f_i(\omega)]\leq 0$$
The prevision $\pi$ is incoherent if for some $n$, $\varepsilon>0$,
$f_1,\dots, f_n \in \mathcal{H}$ and $\beta_1,\dots,\beta_n$ we
have
\begin{equation}\label{book}\sum_{i=1}^n \beta_i (\pi(f_i)-f_i(\omega))> \varepsilon
\quad \forall \omega\in\Omega.
\end{equation}
A combination of gambles which produces an output as in \eqref{book} is called
\textbf{book}.
\end{definition}

\medskip

\begin{remark}\label{corollary:order} It is clear that the prevision is not coherent if, for some $g\in\mathcal{L}^{\infty}$,  $\pi(g)$ is strictly greater
than $\sup_{\omega\in\Omega} g(\omega)$ (or strictly smaller than
$\inf_{\omega\in\Omega} g(\omega)$).
\end{remark}

\medskip

To understand in depth the relation between coherency and No-Arbitrage we consider the case $\mathcal{H}=\{S^{1},\ldots ,S^{d}\}$ where every $S^i\in \mathcal{L}^{\infty}$ is the random payoff of a risky asset at maturity and $\pi(S^i):=s^i$ is the observed initial price. Adopting the nomenclature proposed in \cite{Bu+19}, we recall three different classes of arbitrage opportunities:
\begin{description}
\item[Uniformly strong arbitrage] is a strategy $\beta_1,\dots,\beta_d\in\R^d$ such that
\begin{equation}\label{Ustrong}\sum_{i=1}^d \beta_i (\pi(f_i)-f_i(\omega))> \varepsilon
\quad \forall \omega\in\Omega.
\end{equation}
\item[Strong arbitrage] is a strategy $\beta_1,\dots,\beta_d\in\R^d$ such that
\begin{equation}\label{strong}\sum_{i=1}^d \beta_i (\pi(f_i)-f_i(\omega))> 0
\quad \forall \omega\in\Omega.
\end{equation}
\item[$\mathbb{P}$ arbitrage] is a strategy $\beta_1,\dots,\beta_d\in\R^d$ such that
\begin{eqnarray}\label{classical} && \sum_{i=1}^d \beta_i (\pi(f_i)-f_i(\omega))\geq 0 
\quad \text{for $\mathbb{P}$-almost every } \omega\in\Omega,
\\\nonumber && \mathbb{P} \left(\sum_{i=1}^d \beta_i (\pi(f_i)-f_i(\omega))> 0\right)>0,
\end{eqnarray}
where $\mathbb{P}$ is a reference probability measure on $(\Omega,\CF)$.
\end{description}
We can easily deduce the following implications
\[\exists \text{ a book } \Leftrightarrow \exists \text{ uniformly strong arb. } \Rightarrow \exists \text{ strong arb. } \Rightarrow \exists\, \mathbb{P}\text{-arb,}  \]
no matter which reference probability $\mathbb{P}$ has been chosen in the last implication. 
It is surprising how the notion of coherency matches the notion of \textquoteleft absence of uniformly strong arbitrage opportunity\textquoteright  which recently received great attention in Robust Finance (see \cite[Theorem 3]{Bu+19}). 

\medskip

The following proposition (adopting the formulation in \cite{SSK08}) shows how coherency is naturally related to the existence of a linear pricing functional. 

\medskip

\begin{proposition}\label{linear:pricing} The prevision $\pi:\mathcal{H}\to \R$
is coherent if and only if there exists a positive linear
functional $L:\text{span}(\Hcal)\to \R$ such
that $L(f)=\pi(f)$ for all $f\in\mathcal{H}$ and $L(\mathbf{1}_{\Omega})=1$.
\end{proposition}

\medskip

Even though it is never mentioned in de Finetti's work, the previous result hides some further structural properties of the functional $L$. Therefore we can reformulate Proposition \ref{linear:pricing} in the following more appealing fashion, which relates coherency to the existence of a pricing functional defined through a finitely additive probability measure. 

\medskip

\begin{theorem}\label{pricing:measure} The following conditions are equivalent: 

\begin{enumerate}
    \item The prevision $\pi:\mathcal{H}\to \R$
     is coherent
     \item there exists $\PW\in \mathcal{P}_f$ such that $E_{\PW}[h]=\pi(h)$ for all $h\in\mathcal{H}$
\end{enumerate}

In addition if $\pi$ is coherent, but we can find for some $f_1,\dots, f_n \in \mathcal{H}$ a real vector $\beta_1,\dots,\beta_n$ such that 

\begin{equation}\label{weak:book}\sum_{i=1}^n \beta_i (\pi(f_i)-f_i(\omega))> 0
\quad \forall \omega\in\Omega,
\end{equation}
then $\PW$ cannot be countably additive. 
\end{theorem}

\medskip

The finitely additive probability $\PW$ resulting from Theorem \ref{pricing:measure} can be seen as a martingale measure as soon as we are back to the case $\mathcal{H}=\{S^{1},\ldots ,S^{d}\}$. Nevertheless coherency alone is not sufficient to guarantee $\PW$ to be countably additive, unless the situation in Eq. \ref{weak:book} does not occur (see for example the sufficient conditions provided in \cite[Theorem 3]{Bu+19}).   

\medskip

\begin{example}
 Consider the measurable space $\Omega=(0,1]$ endowed with the Borel sigma algebra $\Fcal$. Let $f(\omega)=\omega$ for any $\omega\in \Omega$ and assume $\pi(f)=0$ while $\pi(\ind_{\Omega})=1$. Indeed $f(\omega)-\pi(f)>0$ for any $\omega\in\Omega$, but $\pi$ is nevertheless coherent. In this situation it is not possible to find a (countably additive) probability measure $Q$ such that $E_{Q}[f]=0$. Nevertheless by Theorem \ref{pricing:measure} there exists a linear pricing functional defined through a finitely additive measure $\PW\in ba_+$. Such a measure necessarily gives mass $1$ only to sets $A\in\Fcal$ such that $\inf_{\omega\in A} f(\omega)=0$ and $0$ otherwise, so that $E_{\PW}[f]=0$.      
\end{example}

\medskip

We conclude this section providing an enhanced version of the so called
Fundamental Theorem of Prevision proved by de Finetti in \cite[Section 3.10]{deF} and we further elaborate a natural pricing-hedging duality. Theorem \ref{fund:prev}
is as far as we know the first result which interconnects superhedging and subhedging
prices (upper and lower previsions) to arbitrage opportunities
(coherence of the prevision). In this way de Finetti was suggesting a notion of No-Arbitrage interval for pricing in incomplete markets. 
\\ We argued so far that the prevision $\pi:\mathcal{H}\to \R$ should be interpreted as the way prices are assessed in a theoretical market. As such the value $\pi(f)$ is observable only in the set $\mathcal{H}$ and its linear extension $L:\text{span}(\Hcal)\to \R$ induces prices only to those positions which can be replicated by a strategy on $\mathcal{H}$. For this reason any element $g\in \mathcal{L}^{\infty}\setminus
\text{span}(\mathcal{H})$ necessitates for a criterion to be priced in a way that coherency is not violated.

\begin{theorem}[Fundamental Theorem of Prevision revisited] \label{fund:prev} Let $\pi:\mathcal{H}\to \R$ be a coherent prevision function, $L$ its linear extension to $\text{span}(\Hcal)$ and $g\in \mathcal{L}^{\infty}\setminus
\text{span}(\mathcal{H})$. Let $g_0\in \R$ be the prevision (or price) assigned to $g$ and define
\begin{eqnarray*}
\underline{\pi}(g) & = & \sup\{L(f)\mid f\in
\text{span}(\mathcal{H}),\; f\leq g\}
\\\overline{\pi}(g) & = & \inf\{L(f)\mid f\in
\text{span}(\mathcal{H}),\; f\geq g\}
\end{eqnarray*}
Then $g_0$ does not violate coherence in $\text{span}(\Hcal\cup \{g\})$ if and only if $g_0$ belongs to
the closed interval $[\underline{\pi}(g),\overline{\pi}(g)]$.
\\ Furthermore\footnote{The following result cannot be credited neither to de Finetti \cite{deF} not to \cite{SSK08}. Nevertheless we consider fruitful to complete the statement with the natural pricing hedging duality.}
\begin{eqnarray}
\underline{\pi}(g) & = & \inf\{E_{\PW}[g]\mid \PW\in \Mcal\} \nonumber
\\\overline{\pi}(g) & = & \sup\{E_{\PW}[g]\mid \PW\in \Mcal\} \label{super:dual}
\end{eqnarray}
where $\Mcal=\{\PW\in \mathcal{P}_f\mid E_{\PW}[h]=\pi(h) \;\forall\,h\in \mathcal{H}\}$.
\end{theorem}

\subsection{Proofs of Section \ref{1980}}
Before proving Proposition \ref{linear:pricing}\footnote{The proof is omitted in \cite{SSK08} as it is quite straightforward. Nevertheless we decided to include it because the spirit of this paper is to provide a complete and self contained review of the topic.}  we state and prove a useful lemma relating coherence to the property that every linear relation which holds for random
payoffs necessarily holds for their prevision as well.

\begin{lemma}\label{coherence:inequality}  The prevision $\pi:\mathcal{H}\to \R$
is coherent if and only for any $\{f_i\}_{i=1}^n\subseteq
\mathcal{H}$ and $\alpha_i\in\R$ such that $\sum_{i=1}^n \alpha_i
f_i\geq c$ we also have $\sum_{i=1}^n \alpha_i \pi(f_i)\geq c$.
\end{lemma}

\begin{proof} We show both directions by contradiction. Assume
first that $\pi$ is not coherent. There exist $n\in\mathbb{N}$, $\{f_i\}_{i=1}^n\subseteq \mathcal{H}$,  $\{\alpha_i\}_{i=1}^n \subset \R$ and
$\varepsilon >0$ such that $\sum_{i=1}^n \alpha_i (f_i-\pi(f_i))>
\varepsilon$. As a consequence $\sum_{i=1}^n \alpha_i f_i \geq
\sum_{i=1}^n \alpha_i \pi(f_i)+\varepsilon$ and by assumption
$\sum_{i=1}^n \alpha_i \pi(f_i) \geq \sum_{i=1}^n \alpha_i
\pi(f_i)+\varepsilon$ which leads to $\varepsilon\leq 0$, a contradiction.
\\For the reverse implication assume that there exist $\{f_i\}_{i=1}^n\subseteq
\mathcal{H}$ and $\alpha_i\in\R$ such that $\sum_{i=1}^n \alpha_i
f_i\geq c$, but $\sum_{i=1}^n \alpha_i \pi(f_i)< c$. Let $0
<\varepsilon< c-\sum_{i=1}^n \alpha_i \pi(f_i)$. We conclude that
$\sum_{i=1}^n \alpha_i (f_i-\pi(f_i))\geq c-\sum_{i=1}^n \alpha_i
\pi(f_i)>\varepsilon$, so that we find a book, prejudicing coherence.
\end{proof}

\begin{proof}[Proof of Proposition \ref{linear:pricing}] Assume that $\pi$ is coherent. We define $L:\text{span}(\mathcal{H})\to \R$ as $L(f)=\sum_{i=1}^n \alpha_i
\pi(f_i)$ if $f= \sum_{i=1}^n \alpha_i f_i$ with $\{f_i\}_{i=1}^n\subset
\mathcal{H}$. First we show that $L$ is well defined. Consider an
alternative representation $f=\sum_{j=1}^m \beta_j g_j$ with
$\{g_j\}_{j=1}^m\subset \mathcal{H}$ and assume $\sum_{i=1}^n \alpha_i
\pi(f_i)> \sum_{j=1}^m \beta_j \pi(g_j)$ (respectively $<$). Then for
$0<\varepsilon<\sum_{i=1}^n \alpha_i \pi(f_i)- \sum_{j=1}^m \beta_j
\pi(g_j)$ we have

$$\sum_{j=1}^m \beta_j (g_j-\pi(g_j))-\sum_{i=1}^n \alpha_i (f_i-\pi(f_i))=f-f+\sum_{i=1}^n \alpha_i \pi(f_i)- \sum_{j=1}^m
\beta_j \pi(g_j)>\varepsilon>0,$$
which is of course a contradiction as $\pi$ is coherent. $L$ is therefore well defined and linear by definition. Moreover
Lemma \ref{coherence:inequality} implies $L(f)\geq 0$
whenever $f\geq 0$ (i.e. $L$ is positive). Finally notice that
$L(\mathbf{1}_{\Omega})\neq 1$ then we can easily build up a book (same as in the proof of Theorem \ref{DF31}).

\medskip

The reverse implication follows immediately from the positivity of
$L$ and Lemma \ref{coherence:inequality}. In fact the existence of a book would lead to $\{f_i\}_{i=1}^n\subseteq \mathcal{H}$,  $\{\alpha_i\}_{i=1}^n \subset \R$ and
$\varepsilon >0$ such that $\sum_{i=1}^n \alpha_i (f_i-\pi(f_i))>
\varepsilon$ but $L\left(\sum_{i=1}^n \alpha_i (f_i-\pi(f_i))\right)=0$ and $L(\varepsilon\cdot\ind_{\Omega})= \varepsilon$.
\end{proof}

\begin{proof}[Proof of Theorem \ref{pricing:measure}] 

    
    1. $\Rightarrow$ 2.: Let $L:\text{span}(\Hcal)\to \R$ be the extension provided in Proposition \ref{linear:pricing}. Indeed positivity of $L$ implies Lipschitz continuity and hence norm continuity. This can be easily seen by the standard argument: $f-g-\|f-g\|_{\infty}\leq 0$ implies $L(f-g)\leq \|f-g\|_{\infty}\cdot L(\mathbf{1}_{\Omega})$. 
    \\ $\text{span}(\mathcal{H})$ is a linear subspace of $\mathcal{L}^{\infty}=\{f:\Omega\to \R\mid \CF\text{-measurable and bounded}\}$, which is on its own an ordered topological vector space (w.r.t. the pointwise order and the norm $\|\cdot\|_{\infty}$). Further the element $g(\omega)\equiv\varepsilon>0$ is an interior point of the ordering cone $\{f\in \mathcal{L}^{\infty}\mid f\geq 0\}$ and belongs to $C\cap \text{span}(\mathcal{H})$. Therefore \cite[Corollary 2, Section 5.4]{Schaefer71} implies that $L$ admits a continuous and positive linear extension $\hat{L}$ to the entire space $\mathcal{L}^{\infty}$.  
    \\ As the topological dual of $\mathcal{L}^{\infty}$ is $\mathbf{ba}$, there exists $\mu:\Fcal\to [0,+\infty)$ which is finitely additive and $\hat{L}(f)=\int_{\Omega}fd\mu$ where the integral is constructed in \cite[Section 11.2]{AlBor}. From $1=L(\ind_{\Omega})=\hat{L}(\ind_{\Omega})=\mu(\Omega)=1$ there exists $\PW\in \mathcal{P}_f$ such that $\hat{L}(f)=E_{\PW}[f]$ for all $f\in\mathcal{L}^{\infty}$, so that $E_{\PW}[h]=\pi(h)$ for all $h\in\mathcal{H}$.
 \medskip

 2. implies 1. By contradiction assume that we can find $n\in\mathbb{N}$,
$f_1,\dots, f_n \in \mathcal{H}$ and $\beta_1,\dots,\beta_n$ such that
\[f(\omega)=\sum_{i=1}^n \beta_i (\pi(f_i)-f_i(\omega))> \varepsilon > 0
\quad \forall \omega\in\Omega.\]
Then simultaneously $E_{\PW}[f]\geq \varepsilon>0$ and $0=L(f)=E_{\PW}[f]$, hence a contradiction. 
\\ The final statement follows immediately from the property that for any countable probability measure $\PW$ and $g\in\mathcal{L}^{\infty}$, $g(\omega)>0$ for any $\omega\in\Omega$, we have $E_{\PW}[g]>0$.   
\end{proof}

\begin{proof}[Proof of Theorem \ref{fund:prev}]
First we show that if $g_0< \underline{\pi}(g)$ then there exists
a book. Let $\varepsilon:= \underline{\pi}(g)-g_0$ and $f\in
\text{span}(\mathcal{H})$ such that $f\leq g$ and
$\underline{\pi}(g)-\pi(f)<\frac{\varepsilon}{2}$. Since $g-f\geq 0$
we have $g-g_0-(f-\pi(f))>\frac{\varepsilon}{2}$, which forms a
book. Similar for $g_0> \overline{\pi}(g)$ so that we have shown
sufficiency.
\\Now we consider the reverse implication: assume that $g_0\in
[\underline{\pi}(g),\overline{\pi}(g)]$, but there exists a book

$$\sum_{i=1}^n \alpha_i (f_i-\pi(f_i))+\alpha(g-g_0)>\varepsilon >0.$$

Necessarily $\alpha\neq 0$, as by assumption we cannot create a book with the elements of $\mathcal{H}$ only. Consider first the case $\alpha>0$ and observe that
$$g\geq \frac{\varepsilon}{\alpha}+g_0+\sum_{i=1}^n \frac{\alpha_i}{\alpha}\pi(f_i) -\sum_{i=1}^n \frac{\alpha_i}{\alpha}f_i\in \text{span}(\mathcal{H}),$$
since we are assuming $\mathbf{1}_{\Omega}\in \mathcal{H}$.
Moreover the prevision $\pi$ applied to
$\frac{\varepsilon}{\alpha}+g_0+\sum_{i=1}^n
\frac{\alpha_i}{\alpha}\pi(f_i) -\sum_{i=1}^n
\frac{\alpha_i}{\alpha}f_i$ gives the value
$\frac{\varepsilon}{\alpha}+ g_0$. This leads to a contradiction
since
\begin{eqnarray*} g_0&\geq & \sup\{\pi(f)\mid f\in
\text{span}(\mathcal{H}),\; f\leq g\}
\\ &\geq & \pi\left(\frac{\varepsilon}{\alpha}+g_0+\sum_{i=1}^n
\frac{\alpha_i}{\alpha}\pi(f_i) -\sum_{i=1}^n
\frac{\alpha_i}{\alpha}f_i\right)= 
\frac{\varepsilon}{\alpha}+ g_0.
\end{eqnarray*}
Consider now the case $\alpha<0$ and observe that
$$g\leq \frac{\varepsilon}{\alpha}+g_0+\sum_{i=1}^n \frac{\alpha_i}{\alpha}\pi(f_i) -\sum_{i=1}^n \frac{\alpha_i}{\alpha}f_i\in \text{span}(\mathcal{H}).$$
This leads to a contradiction
from $\frac{\varepsilon}{\alpha}<0$ jointly with
\begin{eqnarray*} g_0&\leq & \inf\{\pi(f)\mid f\in
\text{span}(\mathcal{H}),\; f\geq g\}
\\ &\leq & \pi\left(\frac{\varepsilon}{\alpha}+g_0+\sum_{i=1}^n
\frac{\alpha_i}{\alpha}\pi(f_i) -\sum_{i=1}^n
\frac{\alpha_i}{\alpha}f_i\right) =
\frac{\varepsilon}{\alpha}+ g_0.
\end{eqnarray*}

\medskip

The proof of the duality \eqref{super:dual} follows from the standard Fenchel Moreu Theorem and we here give a brief sketch for the sake of completeness (referring to \cite{FR02} for the details). We only discuss the case for $\overline{\pi}$ as the other follows immediately from the relation $\underline{\pi}(g)=-\overline{\pi}(-g)$.
\\ We first introduce the set 
$$C=\{g \in \mathcal{L}^{\infty} \mid  g\leq f \text{ for some } f\in \text{span}(\mathcal{H}) \text{ with } L(f)=0 \}$$ 
and notice that $C$ is convex and $\overline{\pi}(g)=\inf\{a\in\R\mid g-a\in C\}$. This implies that $\overline{\pi}$ is a convex monotone functional and $\overline{\pi}(g+a)=\overline{\pi}(g)+a$ for any $a\in\R$. As $C$ is a cone we also have $\overline{\pi}(\alpha g)=\alpha \overline{\pi}(g)$ for any $\alpha>0$. From monotonicity $f\leq g+\|f-g\|_{\infty}$ implies $\overline{\pi}(f)\leq \overline{\pi}(g+\|f-g\|_{\infty})$ so that $\overline{\pi}(f)-\overline{\pi}(g)\leq \|f-g\|_{\infty}$. Inverting the role of $f,g$ we have that $\overline{\pi}$ is Lipschitz continuous and hence norm continuous.
Applying Fenchel Moreou duality theorem in the form of \cite[Corollary 7]{FR02} we find the representation 
\[\overline{\pi}(g)=\sup \{E_{\PW}[g]\mid \PW\in\mathcal{P}_f \text{ s.t. } \overline{\pi}^{\ast}(\PW)<+\infty \}\]
where $\overline{\pi}^{\ast}(\PW)=\sup_{f\in \mathcal{L}^{\infty}}\{E_{\PW}[f]-\overline{\pi}(f)\}$ takes only values $0$ or $+\infty$. We observe that for any $\PW\in\mathcal{P}_f$ such that $\overline{\pi}^{\ast}(\PW)<+\infty$ we have
$\overline{\pi}(g)\geq E_{\PW}[g]$ for any $g\in\mathcal{L}^{\infty}$. Moreover for any $h\in \text{span}(\mathcal{H})$ we know that $\overline{\pi}(h)=L(h)$ so that for $g\in C$ we find $h\in\mathcal{H}$ such that $g\leq h$ and  
\[E_{\PW}[g]\leq L(h)=\overline{\pi}(h)=0, \text{ for } \PW\in\mathcal{P}_f \text{ s.t. } \overline{\pi}^{\ast}(\PW)<+\infty. \]
To conclude we have that for any $\PW\in\mathcal{P}_f$ such that $\overline{\pi}^{\ast}(\PW)<+\infty$ the inequality  
$E_{\PW}[g]\leq 0$ holds for any $g\in C$ and hence necessarily $E_{\PW}[h]=0$ for any $h\in\text{span}(\mathcal{H})\subseteq C$.
\end{proof}

\bibliographystyle{abbrv} 
 \bibliography{RTAP.bib}

\end{document}